\documentclass[12pt]{amsart}

\RequirePackage{mathtools}
\RequirePackage{amsopn}
\RequirePackage{amsfonts}
\RequirePackage{paralist}
\RequirePackage{amssymb}
\RequirePackage{amsthm}
\RequirePackage{mathrsfs}
\usepackage[alphabetic]{amsrefs}
\renewcommand\MR[1]{\relax} % Must follow amsrefs
\usepackage{tikz}
\usetikzlibrary{cd,decorations.pathmorphing,decorations.markings}
\mathtoolsset{showonlyrefs}
\newtheorem{thm}{Theorem}[section] % Experiment
\numberwithin{equation}{section}

\newtheorem{cor}[thm]{Corollary}
\newtheorem{lemma}[thm]{Lemma}
\newtheorem{prop}[thm]{Proposition}

\theoremstyle{definition}
\newtheorem{definition}[thm]{Definition}

\theoremstyle{remark}
\newtheorem{remark}[thm]{Remark}
\newtheorem{example}[thm]{Example}

\newtheorem{mycomment}[thm]{Comment}
{\end{mycomment}\endgroup}
\def\mathcs{C^{*}}
\newcommand{\cs}{\ensuremath{\mathcs}}

\DeclareMathSymbol{\rtimes}{\mathbin}{AMSb}{"6F}

\newcommand\R{\mathbf{R}}

\newcommand\T{\mathbf{T}}
\newcommand\Z{\mathbf{Z}}

\newcommand\set[1]{\{\,#1\,\}}
\newcommand\sset[1]{\{#1\}}

\def\restr#1{|_{{#1}}}
\makeatletter
\def\labelenumi{\textnormal{(\@alph\c@enumi)}}
\def\theenumi{\@alph \c@enumi}
\def\labelenumii{\textnormal{(\@roman\c@enumii)}}
\def\theenumii{\@roman \c@enumii}
\newcount\charno
\def\alphapart#1{\charno=96
\advance\charno by#1\char\charno}

\makeatother

\def\<{\langle}
\def\>{\rangle}
\let\ipscriptstyle=\scriptscriptstyle
\def\lipsqueeze{{\mskip -3.0mu}}
\def\ripsqueeze{{\mskip -3.0mu}}
\def\ipcomma{\nobreak\mathrel{,}\nobreak}
\newbox\ipstrutbox
\setbox\ipstrutbox=\hbox{\vrule height8.5pt% depth 3.5pt
width 0pt}
\def\ipstrut{\copy\ipstrutbox}
\def\lip#1<#2,#3>{\mathopen{\relax_{\ipstrut\ipscriptstyle{
#1}}\lipsqueeze
\langle} #2\ipcomma #3 \rangle}
\def\blip#1<#2,#3>{\mathopen{\relax_{\ipstrut
\ipscriptstyle{ #1}}\lipsqueeze\bigl\langle} #2\ipcomma #3 \bigr\rangle}
\def\rip#1<#2,#3>{\langle #2\ipcomma #3
\rangle_{\ripsqueeze\ipstrut\ipscriptstyle{#1}}}
\def\brip#1<#2,#3>{\bigl\langle #2\ipcomma #3
\bigr\rangle_{\ripsqueeze\ipstrut\ipscriptstyle{#1}}}
\def\angsqueeze{\mskip -6mu}
\def\smangsqueeze{\mskip -3.7mu}
\def\trip#1<#2,#3>{\langle\smangsqueeze\langle #2\ipcomma #3
\rangle\smangsqueeze\rangle_{\ripsqueeze\ipstrut\ipscriptstyle{#1}}}
\def\btrip#1<#2,#3>{\bigl\langle\angsqueeze\bigl\langle #2\ipcomma
#3
\bigr\rangle
\angsqueeze\bigr\rangle_{\ripsqueeze\ipstrut\ipscriptstyle{#1}}}
\def\tlip#1<#2,#3>{\mathopen{\relax_{\ipstrut\ipscriptstyle{
#1}}\lipsqueeze \langle\smangsqueeze\langle} #2\ipcomma #3
\rangle\smangsqueeze\rangle}
\def\btlip#1<#2,#3>{\mathopen{\relax_{\ipstrut\ipscriptstyle{
#1}}\lipsqueeze
\bigl\langle\angsqueeze\bigl\langle} #2\ipcomma #3
\bigr\rangle\angsqueeze\bigr\rangle}

\def\ip(#1|#2){(#1\mid #2)}
\def\bip(#1|#2){\bigl(#1 \mid #2\bigr)}
\def\Bip(#1|#2){\Bigl( #1 \bigm| #2 \Bigr)}
\expandafter\ifx\csname BibSpec\endcsname\relax\else
\BibSpec{collection.article}{%
    +{}  {\PrintAuthors}                {author}
    +{,} { \textit}                     {title}
    +{.} { }                            {part}
    +{:} { \textit}                     {subtitle}
    +{,} { \PrintContributions}         {contribution}
    +{,} { \PrintConference}            {conference}
    +{}  {\PrintBook}                   {book}
    +{,} { }                            {booktitle}
    +{,} { }                            {series}
    +{,} { \voltext}                    {volume}
    +{,} { }                            {publisher}
    +{,} { }                            {organization}
    +{,} { }                            {address}
    +{,} { \PrintDateB}                 {date}
    +{,} { pp.~}                        {pages}
    +{,} { }                            {status}
    +{,} { \PrintDOI}                   {doi}
    +{,} { available at \eprint}        {eprint}
    +{}  { \parenthesize}               {language}
    +{}  { \PrintTranslation}           {translation}
    +{;} { \PrintReprint}               {reprint}
    +{.} { }                            {note}
    +{.} {}                             {transition}
}
\BibSpec{article}{%
    +{}  {\PrintAuthors}                {author}
    +{,} { \textit}                     {title}
    +{.} { }                            {part}
    +{:} { \textit}                     {subtitle}
    +{,} { \PrintContributions}         {contribution}
    +{.} { \PrintPartials}              {partial}
    +{,} { }                            {journal}
    +{}  { \textbf}                     {volume}
    +{}  { \PrintDatePV}                {date}
    +{,} { \eprintpages}                {pages}
    +{,} { }                            {status}
    +{,} { \PrintDOI}                   {doi}
    +{,} { available at \eprint}        {eprint}
    +{}  { \parenthesize}               {language}
    +{}  { \PrintTranslation}           {translation}
    +{;} { \PrintReprint}               {reprint}
    +{.} { }                            {note}
    +{.} {}                             {transition}
}
\BibSpec{book}{%
    +{}  {\PrintPrimary}                {transition}
    +{,} { \textit}                     {title}
    +{.} { }                            {part}
    +{:} { \textit}                     {subtitle}
    +{,} { \PrintEdition}               {edition}
    +{}  { \PrintEditorsB}              {editor}
    +{,} { \PrintTranslatorsC}          {translator}
    +{,} { \PrintContributions}         {contribution}
    +{,} { }                            {series}
    +{,} { \voltext}                    {volume}
    +{,} { }                            {publisher}
    +{,} { }                            {organization}
    +{,} { }                            {address}
    +{,} { pp.~}                        {pages}
    +{,} { \PrintDateB}                 {date}
    +{,} { }                            {status}
    +{}  { \parenthesize}               {language}
    +{}  { \PrintTranslation}           {translation}
    +{;} { \PrintReprint}               {reprint}
    +{.} { }                            {note}
    +{.} {}                             {transition}
}
\fi
\evensidemargin=\oddsidemargin
\newcommand\go{G^{(0)}}

\newcommand\lt{\operatorname{lt}}
\newcommand\rt{\operatorname{rt}}

\newcommand\gx{G[X]}

\newcommand\ho{H^{(0)}}

\newcommand\Iso{\operatorname{Iso}}

\newcommand\so{\Sigma_{0}}

\newcommand\cphi{\mathfrak{c}}

\newcommand\ZT{Z_{\T}}

\begin{document}
\begin{abstract}
  We show that groupoid equivalence preserves a number of groupoid
  properties such as properness or the property of being topologically
  principal.
\end{abstract}

\title{Properties preserved by groupoid equivalence}

\author[van Wyk]{Daniel W. van Wyk}
\address{Department of Mathematics\\ Dartmouth College \\ Hanover, NH
  03755-3551 USA}
\email{dwvanwyk79@gmail.com}

\author[Williams]{Dana P. Williams}
\address{Department of Mathematics\\ Dartmouth College \\ Hanover, NH
  03755-3551 USA}
\email{dana.williams@Dartmouth.edu}

\maketitle
%\makeatletter\providecommand\@dotsep{5}\makeatother\listoftodos\relax

%%%%%%%%%%%%%%%%%%%%%%%%%%%%%
%%% Body

\section*{Introduction}
\label{sec:introduction}

The \emph{raison d'\^etre} for groupoid equivalence is that equivalent
second countable groupoids with Haar systems have Morita equivalent
groupoid \cs-algebras \cite{mrw:jot87}.  Then we can take advantage of
the many properties of \cs-algebras preserved by Morita equivalence
\cites{zet:am82,hrw:pams07}.  This has implications for the
  groupoids themselves.  For example, if $G$ and $H$ are equivalent
second countable groupoids with Haar systems, then $\cs(G)$ is GCR if
and only if $\cs(H)$ is GCR \cite{zet:am82}*{Proposition~3.2}.
However, Clark and the first author have shown that $\cs(G)$ is GCR if
and only if each $G$-orbit in $\go$ is locally closed and every
isotropy group is GCR \citelist{\cite{cla:iumj07}*{Theorem~7.1}
  \cite{wyk:jot18}*{Theorem~4.2}}.  Hence the same must be true for
any second countable groupoid $H$ with a Haar system that is
equivalent to $G$.

In this article, we want to examine properties of groupoids that are
preserved by groupoid equivalence.   Unlike the example above, our
methods do not go through
\cs-theory.  Hence we can separate the property that orbits are
locally closed or the type of the isotropy groups and see that these
are preserved individually.   We can also study properties that have
important implications to the structure of the groupoid \cs-algebra,
such as the groupoid being a proper groupoid, and show that these
properties are preserved as well.  Since we don't use \cs-theory,
we can avoid requiring that our groupoids are second countable and/or
that they have Haar systems.  Nevertheless our methods do
require that groupoids have open range and source maps.  This is
clearly the most important class of groupoids as it is implied by the
existence of a Haar system.

Our paper is organized as follows.  In Section~\ref{sec:preliminaries}
we establish our definitions and recall the notion of the blow-up of a
groupoid.  Since two groupoids are equivalent if and only if they have
isomorphic blow-ups, this allows us to reduce many questions to the
case of a blow-up.

In Section~\ref{sec:prop-cart-group}, we show that groupoid
equivalence preserves proper groupoids as well as Cartan groupoids.
We also 
observe that if $G$ and $H$ are equivalent, then every isotropy group
of $G$ is isomorphic to an isotropy group of $H$ and vice versa.
Hence if $G$ has abelian isotropy, then so does any groupoid $H$
equivalent to $G$.

In Section~\ref{sec:topol-essent-princ} we show that equivalence
preserves the property that a groupoid is topologically principal in
that the set of points $u\in\go$ where the isotropy group,
$G(u)$, is trivial is dense in the unit space.   We show this is also
true for the property of being essentially principal meaning points
with trivial isotropy are dense in every closed invariant subset of
the unit space.  Both these assertions also hold when we replace
trivial isotropy with the  stronger notion of discretely trivial
isotropy introduced by Renault in \cite{ren:jot91}.

In Section~\ref{sec:orbits-isotropy}, we see that the orbit spaces of
equivalent groupoids are homeomorphic and what equivalence implies
about the continuity of the isotropy map $u\mapsto G(u)$.  We apply
this to show that the property of being proper modulo the isotropy
introduced in \cite{wykwil:iumj22} is preserved by equivalence.

In Section~\ref{sec:integrable-groupoids} we examine how the concept of
integrablity, introduced by Clark and an Huef, behaves under
equivalence.

In Section~\ref{sec:t-groupoids}, we recall and sharpen some results from
\cite{muhwil:plms395}*{\S3} about equivalence of twists.

\subsubsection*{Assumptions}
\label{sec:assumptions}

Here groupoid always means a locally compact Hausdorff groupoid with
open range and source maps.  An isomorphism of groupoids means an
isomorphism of topological groupoids; that is, an algebraic
isomorphism that is also a homeomorphism.

\section{Preliminaries}
\label{sec:preliminaries}

Groupoid actions are discussed in detail in
  \cite{wil:toolkit}*{\S2.1}. Here we recall that in order for a
groupoid $G$ to act on a space $P$ we require a continuous
\emph{moment map} $r_{P}:P\to \go$ so that $\gamma\cdot p$ is defined
exactly when $s(\gamma)=r_{P}(p)$.  In most situations, we drop the
subscript ``$P$'' and simply write $s(\gamma)=r(p)$ and trust the
meaning is clear from context.  We say that the action is \emph{free}
if $\gamma\cdot p=p$ if and only if $\gamma=r(p)$ and that the action
is \emph{proper} if the map $\Theta:G*P\to P\times P$ given by
$(\gamma,p)\mapsto (\gamma\cdot p,p)$ is a proper map in that
$\Theta^{-1}(K)$ is compact whenever $K\subset P\times P$ is compact.

In many places in the literature, moment maps for groupoid actions are
also required to be open.  For example, this is the case in
\cite{mrw:jot87} and \cite{simwil:jot11}.  One reason for this is that
it is an appropriate assumption in the definition of groupoid
equivalence.

\begin{definition}[\cite{wil:toolkit}*{Definition~2.29}]
  \label{def-equivalence}
  Suppose that $G$ and $H$ are locally compact Hausdorff
  groupoids. Then a locally compact Hausdorff space $Z$ is called a
  \emph{$(G,H)$-equivalence} if the following conditions are
  satisfied.
  \begin{enumerate}[\enspace (E1)]
  \item $Z$ is a free and proper left $G$-space.
  \item $Z$ is a free and proper right $H$-space.
  \item The $G$- and $H$-actions commute.
  \item The moment map $r_{Z}$ is open and induces a homeomorphism of
    $Z/H$ onto $\go$.
  \item The moment map $s_{Z}$ is open and induces a homeomorphism of
    $G\backslash Z$ onto $\ho$.
  \end{enumerate}
\end{definition}
\begin{remark}
  Definition~\ref{def-equivalence} could be used even when $G$ and $H$
  aren't required to have open range and source maps.  However, when
  $G$ and $H$ do have open range and source maps, it follows from
  \cite{wil:toolkit}*{Remark~2.31} that we can replace (E4) and (E5)
  with the assertion that the moment maps $r_{Z}$ and $s_{Z}$ are open
  and simply induce bijections.  Since moment maps were always assumed
  to be open in \cite{mrw:jot87} and \cite{simwil:jot11},
  Definition~\ref{def-equivalence} agrees with
  \cite{mrw:jot87}*{Definition~2.1} and
  \cite{simwil:jot11}*{Definition~1.4}.
\end{remark}

Since we sometimes construct new groupoids from existing ones, it is
necessary to take some care to verify that the new range and source
maps are open.  As a bit of comfort, we including the following
observation which is definitely in the spirit of the central theme of
this paper.

\begin{lemma}
  Suppose that $G$ and $H$ are equivalent locally compact Hausdorff
  groupoids that do not necessarily have open range and source maps.
  If $H$ has open range and source maps, then so does $G$.
\end{lemma}
\begin{proof}
  Let $Z$ be a $(G,H)$ equivalence. It suffices to see that
  $r:G\to\go$ is open. To do this, we use Fell's Criterion
  \cite{wil:toolkit}*{Proposition~1.1}.  Hence if $u_{n}\to r(\gamma)$
  in $\go$, it will suffice to pass to a subnet, relabel, and find
  $\gamma_{n}\to \gamma$ with $r(\gamma_{n})=u_{n}$.  Since $r_{Z}$
  induces a homeomorphism of $Z/H$ onto $\go$, then are $z_{n}\in Z$
  with $r_{Z}(z_{n})=u_{n}$ and $z_{n}\cdot H\to z\cdot H$ in $Z/H$
  with $r_{Z}(z)=r(\gamma)$ Since $H$ has open range and source maps,
  the orbit map $q:Z\to Z/H$ is open
  \cite{wil:toolkit}*{Proposition~2.12}.  Thus we can pass to a
  subnet, relabel, and assume that there are $\eta_{n}\in H$ such that
  $z_{n}\cdot \eta_{n}\to z$ in $Z$.  Replacing $z_{n}\cdot \eta_{n}$
  with $z_{n}$, we have $z_{n}\to z$ with $r_{Z}(z_{n})= u_{n}$ and
  $r_{Z}(z)=r(\gamma)$.

  Since $s_{Z}$ is open by definition, we can pass to another subnet,
  relabel, and assume that there are $w_{n}\to \gamma\cdot z$ in $Z$
  with $s_{Z}(w_{n})=s_{Z}(z_{n})$.  Since $s_{Z}(w_{n})=s_{Z}(z_{n})$
  there is a unique $\gamma_{n}\in G$ such that
  $w_{n}=\gamma_{n}\cdot z_{n}$.  Since $Z$ is a proper $G$-space,
  \cite{wil:toolkit}*{Proposition~2.17} implies that we can pass to a
  subnet and assume that $\gamma_{n}\to \gamma'$ in $G$.  Since
  $(\gamma_{n}\cdot z_{n})$ converges to both $\gamma'\cdot z$ and
  $\gamma\cdot z$, we have $\gamma'=\gamma$.  But
  $r(\gamma_{n})=r_{Z}(z_{n})=u_{n}$ and we're done.
\end{proof}

A well-known and fundamental example of groupoid equivalence is
  provided by the \emph{blow-up} construction.  If $f:X\to\go$ is a
continuous \emph{open}, surjection, then the blow-up is the groupoid
\begin{equation}
  \label{eq:1a}
  \gx=\set{(x,\gamma,y)\in X\times G\times X:\text{$f(x)=r(\gamma)$
      and $s(\gamma)=f(y)$}}
\end{equation}
equipped with the obvious groupoid structure (see
  \cite{wil:toolkit}*{Example~2.37} for details).  Then $G$ and $\gx$
are equivalent groupoids.  The proof requires both the openness of $f$
as well as the assumption that $G$ has open range and source maps.

The unit space of $\gx$ is $\set{(x,f(x),x):x\in X}$ which we often
silently identify with $X$.

An important feature of blow-ups, especially for our needs here, is
the following result.  A proof is given in
\cite{wil:toolkit}*{Theorem~2.52}.
\begin{thm}
  \label{thm-main-tool} Two groupoids $G$ and $H$ are equivalent if
  and only if they have isomorphic blow-ups.
\end{thm}

\begin{remark}
  [Standard Technique] \label{rem-std} In the sequel, we will
  repeatedly use the following observation.  Since topological
  properties of groupoids are clearly preserved by isomorphism, to
  show that such a property is preserved by equivalence of groupoids,
  it follows from Theorem~\ref{thm-main-tool} that it suffices to show
  that $G$ has the property if and only if any blow-up $\gx$ has the
  same property.
\end{remark}

\section{Proper and Cartan Groupoids}
\label{sec:prop-cart-group}

Recall that we call a groupoid $G$ is \emph{proper} if it acts
properly on its unit space. Our next example illustrates that
  interesting proper groupoids exist in abundance.

\begin{example}\label{ex-action-version}
  Suppose that $G$ acts properly on the left of the space $P$.  Let
  $G\rtimes P=\set{(p,\gamma,q)\in P\times G\times P:p=\gamma\cdot q}$
  be the action groupoid as in \cite{wil:toolkit}*{Definition~2.5}.
  Then $G\rtimes P$ is a proper groupoid---see
  \cite{wil:toolkit}*{Ex~2.1.14}.\footnote{Note that this includes the
    assertion that $G\rtimes P$ has open range and source maps even if
    the moment map is not open.  See
    \cite{wil:toolkit}*{Remark~2.14}.}
\end{example}

A slightly weaker notion is that of a \emph{Cartan} groupoid.  In
analogy with \cite{pal:aom61}*{Definition~1.1.2}, a $G$-space $P$ is
called a \emph{Cartan $G$-space} if every point in $P$ has a compact
neighborhood $K$ such that $\Theta^{-1}(K\times K)$ is compact
\cite{wil:toolkit}*{Definition~2.22}.  We call a groupoid $G$ Cartan
if $\go$ is a Cartan $G$-space for the usual $G$-action.  Just as in
Example~\ref{ex-action-version}, if $P$ is a Cartan $G$-space, then
the action groupoid $G\rtimes P$ is a Cartan groupoid.  Clearly, every
proper $G$-space is Cartan, but the converse can fail even for group
actions---see \cite{pal:aom61}*{p.~298}.

We can now state our first permanence result.

\begin{thm}
  \label{thm-proper+cartan} Suppose that $G$ and $H$ are equivalent
  groupoids.
  \begin{enumerate}
  \item $G$ is proper if and only if $H$ is proper.
  \item $G$ is Cartan if and only if $H$ is Cartan.
  \end{enumerate}
\end{thm}

We start with some preliminary observations.  We let
$\pi:G\to\go\times\go$ be given by
$\pi(\gamma)=\bigl(r(\gamma),s(\gamma)\bigr)$.

\begin{lemma}
  \label{lem-seq-proper} A groupoid $G$ is proper if and only if
  whenever a net $(\gamma_{n})\subset G$ is such that
  $\pi(\gamma_{n})\to (u,v)\in\go\times\go$, then $(\gamma_{n})$ has a
  convergent subnet.
\end{lemma}
\begin{proof}
  Since $r(\gamma_{n})=\gamma_{n}\cdot s(\gamma_{n})$, the lemma is
  just a restatement of \cite{wil:toolkit}*{Proposition~2.17(PA3)}.
\end{proof}

\begin{lemma}
  \label{lem-proper-blow} Suppose that $f:X\to\go$ is a continuous,
  open surjection.  Then $G$ is proper if and only if $\gx$ is proper.
\end{lemma}
\begin{proof}
  We use the criterion given in Lemma~\ref{lem-seq-proper}.

  Suppose that $G$ is proper and that
  $\bigl((x_{n},\gamma_{n},y_{n})\bigr) \subset G[X]$ is such that
  $x_{n}\to x$ and $y_{n}\to y$ in $X$, respectively.  Then
  $r(\gamma_{n})\to f(x)$ and $s(\gamma_{n})\to f(y)$.  Since $G$ is
  proper, $(\gamma_{n})$ has a convergent subnet.  But then
  $\bigl((x_{n},\gamma_{n},y_{n})\bigr)$ also has a convergent subnet.
  Hence $G[X]$ is proper.

  Conversely, suppose that $\gx$ is proper and that $(\gamma_{n})$ is
  a net in $G$ such that $\bigl(\pi(\gamma_{n})\bigr)$ converges in
  $\go\times\go$.  Since $f$ is open, we can pass to a subnet,
  relabel, and assume that there are $x_{n }\in X$ such that
  $f(x_{n})=r(\gamma_{n})$ and $x_{n}\to x$ in $X$.  Similarly, we can
  pass to another subnet, relabel, and assume that there are
  $y_{n}\in X$ such that $f(y_{n })=s(\gamma_{n})$ and $y_{n}\to y$.
  Since $G[X]$ is proper and
  $(x_{n},\gamma_{n},y_{n})\cdot y_{n}=x_{n}$, the net
  $\bigl((x_{n},\gamma_{n},y_{n})\bigr)$ must have a convergent
  subnet.  Since the topology on $G[X]$ is the relative product
  topology, $(\gamma_{n})$ must have a converent subnet.  Hence $G$ is
  proper as claimed.
\end{proof}

\begin{remark}
  Although the first part of the proof of Lemma~\ref{lem-proper-blow}
  did not require the openness of $f$, it is still required to see
  that $G$ and $\gx$ are equivalent.
\end{remark}

The following criterion follows from
\cite{wil:toolkit}*{Definition~2.22 and Lemma~2.23}.
\begin{lemma}
  \label{lem-cartan-criterion} A $G$-space $P$ is Cartan if and only
  if every $p\in P$ has a compact neigbhorhood $K$ such that
  \begin{equation}
    \label{eq:2a}
    P(K,K)=\set{\gamma\in G:\gamma\cdot K\cap K\not=\emptyset}
  \end{equation}
  is compact in $G$.
\end{lemma}

\begin{lemma}
  \label{lem-cartan-blow} Suppose that $f:X\to\go$ is a continuous,
  open surjection.  Then $G$ is Cartan if and only if $\gx$ is Cartan.
\end{lemma}
\begin{proof}
  We will use the criterion from Lemma~\ref{lem-cartan-criterion}.

  Suppose that $G$ is Cartan and $x\in X$.  Let $u=f(x)$.  By
  assumption, there is a compact neighborhood $K$ of $u$ such that
  \begin{equation}
    \label{eq:6}
    P(K,K)=\set{\gamma\in G:\text{$s(\gamma)\in K$ and $r(\gamma)\in K$}}
  \end{equation}
  is compact in $G$.  Since $f$ is continuous and open, there is a
  compact neighborhood, $K'$, of $x$ in $X$ such that
  $f(K')=K$.\footnote{Every point $y\in f^{-1}(K)$ has a compact
    neighborhood $K_{y}$.  Hence there are finitely many $y_{i}$ such
    that $K\subset \bigcup_{i}f(K_{y_{i})}$.  Then
    $K'=f^{-1}(K)\cap \bigl(\bigcup K_{y_{i}}\cup K_{x}\bigr)$ is
    compact.}

  Suppose that $\bigl((x_{i},\gamma_{i},y_{i})\bigr)$ is a net in
  $P(K',K')$.  Then $(x_{i})$ and $(y_{i})$ are nets in $K'$ while
  $(\gamma_{i})$ is a net in $P(K,K)$.  Passing to multiple subnets,
  and relabeling, we can assume that
  $(x_{i},\gamma_{i},y_{i}) \to (x,\gamma,y)$ in $\gx$.  Since
  $P(K',K')$ is closed, it must be compact.  Since $x\in X$ was
  arbitrary, $\gx$ is Cartan.

  Conversely, suppose that $\gx$ is Cartan.  Fix $u\in\go$.  Let
  $x\in f^{-1}(u)$, and let $K'$ be a compact neighborhood of $x$ such
  that $P(K',K')$ is compact in $\gx$.  Let $K=f(K')$.  Since $f$ is
  continuous and open, $K$ is a compact neighborhood of $u$.  Let
  $(\gamma_{i})$ be a net in $P(K,K)$.  Therefore, $r(\gamma_{i})$ and
  $s(\gamma_{i})$ are both in $K$.  Let $x_{i},y_{i}\in K'$ be such
  that $f(x_{i})=r(\gamma_{i})$ and $f(y_{i})=s(\gamma_{i})$.  Then
  $\bigl((x_{i},\gamma_{i},y_{i})\bigr)$ is a net in the compact set
  $P(K',K')$.  Therefore $(\gamma_{i})$ has a convergent subnet and
  the closed set $P(K,K)$ must be compact.
\end{proof}

\begin{proof}[Proof of Theorem~\ref{thm-proper+cartan}]
  The theorem now follows immediately from Lemmas
  \ref{lem-proper-blow}~and \ref{lem-cartan-blow} as in
  Remark~\ref{rem-std}.
\end{proof}

Free Cartan spaces play a significant role due to the following
observation.  If $P$ is a free $G$-space, then we let
\begin{equation}
  \label{eq:1b}
  R(G,P)=\set{(p,q)\in P\times P:G\cdot p=G\cdot q}.
\end{equation}
Then since the action is free, there is a well-defined map
$\tau:R(G,P)\to G$, called the \emph{translation map}, such that
$\tau(p,q)\cdot q=p$.

\begin{lemma}
  Let $P$ be a free $G$-space.  Then $P$ is a Cartan $G$-space if and
  only if its translation map is continuous.
\end{lemma}
\begin{proof}
  This is \cite{wil:toolkit}*{Proposition~2.24}.
\end{proof}

Recall that if $u\in\go$, then
$G(u)=\set{\gamma\in G:r(\gamma)=u=s(\gamma)}$ is called the
\emph{isotropy group} of $G$ at $u$. We need the following which will
be used repeatedly in the sequel.

\begin{prop}
  \label{prop-iso-iso} Suppose that $Z$ is a $(G,H)$-equivalence.
  Then the isotropy groups $G(r(z))$ and $H(s(z))$ are isomorphic.
\end{prop}
\begin{proof}
  This is \cite{wil:toolkit}*{Ex~2.3.7}.  We supply a proof for
  convenience. Let $u=r(z)$ and $v=s(z)$.  Define $\phi:G(u)\to H(v)$
  by $\phi(g)=\tau_{H}(z,g\cdot z)$ where $\tau_{H}$ is the
  translation map for the free and proper right $H$-action on $Z$ (see
  \cite{wil:toolkit}*{Lemma~2.42}).  Then $\phi$ is continuous and
  $\phi(g)$ is the unique element of $H(v)$ such that
  $z\cdot \phi(g)=g\cdot z$.  Therefore $\phi$ is injective and
  $z\cdot \phi(gg')=(gg')\cdot z=g\cdot (g'\cdot z)=g\cdot (z\cdot
  \phi(g'))=(g\cdot z)\cdot \phi(g')=z\cdot (\phi(g)\phi(g')$.
  Therefore $\phi$ is a homomorphism.  We can also define a continuous
  map $\psi:H(v)\to G(u)$ by $\psi(h)=\tau_{G}(z\cdot h,z)$.  It is
  not hard to check that $\psi=\phi^{-1}$.
\end{proof}

Recall that a groupoid $G$ is principal if the action of $G$ on $\go$
is free. It follows from Proposition~\ref{prop-iso-iso} that
equivalence preserves principal groupoids.  Moreover, if $G$ is a
principal Cartan groupoid, then $\go$ is a free Cartan $G$-space.
Hence using Proposition~\ref{prop-iso-iso}, we get the following
corollary of Theorem~\ref{thm-proper+cartan}(b).

\begin{cor}\label{cor-free-cartan-equi}
  Suppose that $G$ and $H$ are equivalent groupoids.  Then $H$ is a
  principal Cartan groupoid if and only if $G$ is a principal Cartan
  groupoid.
\end{cor}

\begin{remark}
  We can prove Corollary~\ref{cor-free-cartan-equi} without reference
  to Theorem~\ref{thm-proper+cartan}(b) as follows.  If $G[X]$ is a
  blow-up of a principal groupoid $G$, then $\gx$ is principal and the
  two translation maps on the unit spaces are related by
  \begin{equation}
    \label{eq:5}
    \tau_{G[X]}(x,y)=\bigl(x,\tau_{G}\bigl(f(x),f(y)\bigr),y\bigr).
  \end{equation}
  Now we can employ Lemma~\ref{lem-cartan-criterion} and use
  Remark~\ref{rem-std} as usual.
\end{remark}

\section{Topologically and Essentially Principal Groupoids}
\label{sec:topol-essent-princ}

There are many situations where having a dense set of points where the
isotropy is trivial is a crucial hypothesis.  Therefore, we make the
following definitions.

\begin{definition}\label{def-top-ess-prin}
  A groupoid $G$ is \emph{topologically principal} if
  $\set{u\in\go:G(u)=\sset u}$ is dense in $\go$.  We say that $G$ is
  \emph{essentially principal} if for every $G$-invariant closed set
  $F\subset \go$, $\set{u\in F:G(u)=\sset u}$ is dense in $F$.
\end{definition}

\begin{remark}
  \label{rem-chaos} The terminology adopted here is far from being
  universally accepted.  In fact, the literature is rather chaotic in
  naming these properties.  What is called here ``topologically
  principal'' has been called ``topologically free'' in many places
  and certainly that term has been conflated with ``essentially free''
  as well.  The second author used ``essentially free'' in place of
  ``essentially principal'' in \cite{wil:toolkit}*{Chap.~11}.  Renault
  used ``essentially free'' in \cite{ren:jot91}*{Definitioin~4.8} for
  the stronger property that points with discretely trivial isotropy
  are dense in closed invariant sets (which was called ``strongly
  essentially free'' in \cite{wil:toolkit}*{Proposition~11.35}.)
  Unfortunately, it seems ``essentially free'' has a different meaning
  in group dynamics.  Renault also uses the term topologically
  principal in \cite{ren:irms08}*{Definition~3.5}, but topologically
  principal was often called essentially principal in articles such as
  \cite{arzren:oaqft97} and topologically free in
  \cite{tom:interplay92}.
\end{remark}

\begin{remark}\label{rem-red-prin}
  Note that $G$ is essentially principal if and only if the reduction,
  $G(F)=\set{\gamma\in G:\text{$r(\gamma)\in F$ and
      $s(\gamma)\in F$}}$, is topologically principal for every
  $G$-invariant closed subset $F\subset \go$. (Since $F$ is saturated,
  $G(F)$ has open range and source maps.)
\end{remark}

\begin{thm}\label{thm-top-ess-prin}
  Suppose that $G$ and $H$ are equivalent groupoids.
  \begin{enumerate}
  \item $G$ is topologically principal if and only if $H$ is
    topologically principal.
  \item $G$ is essentially principal if and only if $H$ is essentially
    principal.
  \end{enumerate}
\end{thm}

As in the previous section, we will work first with blow-ups.  We will
repeatedly use the observation that if $f:X\to\go$ is a continuous
open surjection, then the isotropy group $G[X](x)$ is trivial if and
only if the isotropy group $G(f(x))$ is trivial.

\begin{lemma}
  \label{lem-gx-top-prin} Suppose that $f:X\to\go$ is a continuous,
  open surjection.  Then $G$ is topologically principal if and only if
  $G[X]$ is topologically principal.
\end{lemma}
\begin{proof}
  If $G$ is topologically principal and $x\in X$, then there are
  $u_{i}\to f(x)$ with $G(u_{i})$ trivial.  Since $f$ is open, we can
  pass to subnet, relabel, and assume that there are $x_{i}\to x$ with
  $f(x_{i})=u_{i}$.  Then $G[X](x_{i})$ is trivial, and it follows
  that $G[X]$ is topologically principal.

  Conversely, suppose that $G[X]$ is topologically principal and that
  $f(x)\in\go$.  Then there is a sequence $(x_{i})\subset X$ with
  $x_{i}\to x$ and each $G[X](x_{i})$ trivial.  But then
  $f(x_{i})\to f(x)$ and each $G(f(x_{i}))$ is trivial.  Therefore $G$
  is topologically principal.
\end{proof}

To get a similar result for essentially principal groupoids, we can
use the observation that if $F\subset \go$ is closed and
$C=f^{-1}(F)$, then the restriction $f\restr C:C\to F$ is open and
\begin{equation}
  \label{eq:1}
  \gx(C)=G(F)[C].
\end{equation}

\begin{cor}
  \label{cor-gx-essen-prin} Suppose that $f:X\to \go$ is a continuus,
  open surjection.  Then $G$ is essentially principal if and only if
  $\gx$ is essentially principal.
\end{cor}
\begin{proof}
  Suppose that $G[X]$ is essentially principal.  Let $F$ be a closed
  $G$-invariant subset of $\go$.  Then $C:=f^{-1}(F)$ is closed.
  Suppose $x\in C$ and $(x,\gamma,y)\in G[X]$.  Then $f(x)=r(\gamma)$
  and $s(\gamma)=f(y)$.  But $r(\gamma)\in F$ implies
  $s(\gamma)\in F$.  Hence $y\in C$.  Therefore $C$ is
  $G[X]$-invariant.  Then by assumption $\gx(C)$ is topologically
  principal.  By \eqref{eq:1}, $G(F)[C]$ is topologically principal
  and $G(F)$ is too by Lemma~\ref{lem-gx-top-prin}.  Since $F$ was
  arbitrary, this shows that $G$ is essentially principal.

  Conversely, suppose that $G$ is essentially principal and that $C$
  is a closed $\gx$-invariant subset of $X$.  We claim that $F:=f(C)$
  is closed in $\go$.  Suppose $f(x_{i})\to u$ with each $x_{i}\in C$.
  Let $u=f(x)$.  Since $f$ is open, we can pass to subnet, relabel,
  and assume that there are $y_{i}\in X$ with $y_{i}\to x$ and
  $f(y_{i})=f(x_{i})$.  But then $(x_{i},f(x_{i}),y_{i})\in G[X]$.
  Since $C$ is $G[X]$-invariant and $x_{i}\in C$, this means each
  $y_{i}\in C$.  Since $C$ is closed, $x\in C$ and $u\in F$.  Thus $F$
  is closed.

  To see that $F$ is $G$-invariant, suppose $r(\gamma)\in F$.
  Therefore there is a $x\in C$ such that $f(x)=r(\gamma)$.  Let $y$
  be such that $f(y)=s(\gamma)$.  Then $(x,\gamma,y)\in G[X]$ and
  $x\in C$.  Since $C$ is $G[X]$-invariant, we must have $y\in C$.
  This means that $s(\gamma)\in F$, and $F$ is a closed $G$-invariant
  subset of $\go$.

  Suppose that $y\in f^{-1}(F)$.  Then there is a $x\in C$ such that
  $f(x)=f(y)$.  Then $(x,f(x),y)\in G[X]$.  Since $C$ is
  $G[X]$-invariant, $y\in C$.  Hence $C=f^{-1}(F)$.

  By assumption, $G(F)$ is topologically principal.  By
  Lemma~\ref{lem-gx-top-prin}, $G(F)[C]$ is too.  By \eqref{eq:1},
  $\gx(C)$ is also topologically principal.  Since $C$ was arbitrary,
  $\gx$ is essentially principal.
\end{proof}

A locally compact Hausdorff \emph{\'etale} groupoid $G$ is called
  \emph{effective} if the interior of
  $\Iso(G)=\set{\gamma\in G:r(\gamma)=s(\gamma)}$ is reduced to $\go$.
  It follows from \cite{ren:irms08}*{Proposition~3.6} that
  topologically principal \'etale groupoids are always effective, and
  the converse holds when $G$ is also second countable.  However, our
  next example shows that the blow-up of an \'etale groupoid need not
  be \'etale.   Hence the property of being \'etale is not
  preserved by equivalence.
  
  \begin{example}[Blow-Ups of \'Etale Groupoids]
    \label{ex-not-etale} Suppose that $f:X\to\go$ is a continuous,
    open surjection such that there is a $u\in\go$ such that
    $f^{-1}(u)$ is not discrete.  Then $\gx$ is not \'etale.  To see
    this note that there is a sequence
    $(x_{n})\subset f^{-1}(u)\setminus \sset x$ such that
    $x_{n}\to x$.  But then $(x_{n},u,x)\to (x,u,x)$ in $\gx$ and
    $\bigl((x_{n},u,x)\bigr)\subset \gx\setminus \gx^{(0)}$.  Hence
    $\gx^{(0)}$ is not open in $\gx$ and $\gx$ is not \'etale.
    We make some additional comments on equivalence and \'etale
      groupoids in an appendix.
  \end{example}

  However we can still make the following observation.

\begin{remark}
  [Effective] If $G$ and $H$ are equivalent \emph{second countable}
  \'etale groupoids, then one is effective if and only the other is by
  Theorem~\ref{thm-top-ess-prin} and the remarks preceding
  Example~\ref{ex-not-etale}.  We do not know whether this holds if we
  drop the second countability assumption.  Our methods break down
  since Example~\ref{ex-not-etale} shows that the blow-up of an
  \'etale groupoid need not be \'etale. 
\end{remark}

In \cite{ren:jot91}, Renault defines the isotropy of $G$ to be
\emph{discretely trivial at $v\in \go$} if for every compact set
$K\subset G$, there is a neighborhood $V$ of $v$ in $\go$ such that
$u\in V$ implies that $K\cap G(u)\subset \sset u$.  It is not hard to
see that if $G$ is discretely trivial at $v$, then $G(v)=\sset v$---if
not, let $K=\sset \gamma$ for $\gamma\in G(v)\setminus \sset v$.
Since every neighborhood of $v$ contains $v$, we obtain a
contradiction.  If $G$ is \'etale, then the converse holds
\cite{wil:toolkit}*{Lemma~11.19(c)}.  However, the converse fails in
general \cite{wil:toolkit}*{Example~11.23}.  The concept of discrete
triviality plays a key role in extending some simplicity and structure
results to non-\'etale groupoids in \cite{ren:jot91}.  (See also
Corollary~\ref{cor-minimal}.)

\begin{prop}
  \label{prop-discretely-trivial} Suppose that $Z$ is a
  $(G,H)$-equivalence.  Then the isotropy of $G$ is discretely trivial
  at $r(z)$ if and only if the isotropy of $H$ is discretely trivial
  at $s(z)$.
\end{prop}

We first prove the result for blow-ups. Define $\cphi:\gx\to G$ by
$\cphi(x,\gamma,y)=\gamma$.  Clearly, $\cphi$ is a groupoid
homomorphism. % \todo{Need to rename $\cphi$ and move to near
% definition of $\gx$.}

\begin{lemma}
  \label{lem-discretely-trivial} Suppose that $f:X\to\go$ is a
  continuous, open surjection.  Then the isotropy of $\gx$ is
  discretely trivial at $y\in X$ if and only if the isotropy of $G$ is
  discretely trivial at $f(y)$.
\end{lemma}
\begin{proof}
  Suppose that $G$ is discretely trivial at $v\in\go$ and $f(y)=v$.
  Let $C$ be a compact set in $\gx$.  Then $K:=\cphi(C)$ is compact
  and there is a neighborhood $V$ of $v$ in $\go$ such that $u\in V$
  implies that $K\cap G(u)\subset \sset u$.  Let $V'=f^{-1}(V)$.  Then
  $V'$ is a neighborhood of $y$.  Suppose that $x\in V'$ and
  $\sigma=(x,\gamma,x)\in C\cap \gx(x)$.  Then
  $\gamma\in G(f(x))\cap K$.  Therefore $\gamma=f(x)$ and
  $\sigma=(x,f(x),x)$.  Since $C$ was arbitrary, $\gx$ is discretely
  trivial at $y$.

  Conversely, suppose that $\gx$ is discretely trivial at $y\in X$.
  Let $v=f(y)$ and suppose to the contrary of what we want to prove
  that $G$ is not discretely trivial at $v$.  Then there is a compact
  set $K$ such that given any neighborhood $V$ of $v$ in $\go$, there
  is $u_{V}\in V$ and $\gamma_{V}\not= u_{V}$ in $K\cap G(u_{V})$.
  After passing to a subnet, and relabeling, we can assume that
  $u_{V}\to v$, $\gamma_{V} \to \gamma \in G(v)$.  Furthermore, since
  $f$ is open, we can also assume there are $x_{V}\to y$ in $X$ with
  $f(x_{V})=u_{V}$.  Then
  $(x_{V},\gamma_{V},x_{V})\in \gx(x_{V})\subset \gx$ and
  $(x_{V},\gamma_{V},x_{V}) \to (y,\gamma,y)$.  Let $C$ be a compact
  neighborhood of $(y,\gamma,y)$ in $\gx$.  Since $\gx$ is discretely
  trivial at $y$ there is a neighborhood $V'$ of $y$ such that
  $x\in V'$ implies that $C\cap \gx(x)\subset \sset{(x,f(x),x)}$.
  This leads to a contradiction as we eventually have
  $(x_{V},\gamma_{V},x_{V})$ in $C$.
\end{proof}

\begin{proof}[Proof of Proposition~\ref{prop-discretely-trivial}]
  Let $G[Z]$ be the blow-up of $G$ with respect to the moment map
  $r_{Z}:Z\to\go$ and let $H[Z]$ be the blow-up of $H$ with respect to
  the moment map $s_{Z}:Z\to\ho$.  Then as in the proof of
  \cite{wil:toolkit}*{Theorem~2.52}, we get an isomorphism
  $\phi:G[Z]\to H[Z]$ given by
  \begin{equation}
    \label{eq:2}
    \phi(z,\gamma,w)=(z,\tau_{H}(z,\gamma\cdot w),w).
  \end{equation}
  Since $\phi(z,r_{Z}(z),z)=(z,s_{Z}(z),z)$, the isotropy of $G[Z]$ is
  discretely trivial at $(z,r_{Z}(z),z)$ if and only the isotropy of
  $H[Z]$ is discretely trivial at $(z,s_{Z}(z),z)$.  Now the
  proposition follows from Lemma~\ref{lem-discretely-trivial}.
\end{proof}

Recall that if $G$ is \'etale, then the isotropy at $v\in\go$ is
discretely trivial if and only if $G(v)=\sset v$
\cite{wil:toolkit}*{Lemma~11.19(c)}.  Hence we obtain the following
Corollary.

\begin{cor}
  \label{cor-discretely-trivial} Suppose that $G$ and $H$ are
  equivalent groupoids such that $H$ is \'etale.  Then the isotropy of
  $G$ is discretely trivial at $u\in\go$ if and only if
  $G(u)=\sset u$.
\end{cor}

\begin{definition}\label{def-str-top-ess-prim}
  A groupoid $G$ is \emph{strongly topologically principal} if the set
  of points with discretely trivial isotropy is dense in $\go$.  We
  say that $G$ is \emph{strongly essentially principal} if the points
  with discretely trivial isotropy are dense in every closed
  $G$-invariant subset of $\go$.
\end{definition}

\begin{remark}
  What we are calling strongly essentially principal here is what
  Renault called essentially free in \cite{ren:jot91}.  Note that if
  $G$ is equivalent to an \'etale groupoid, then
  Definition~\ref{def-str-top-ess-prim} reverts to
  Definition~\ref{def-top-ess-prin} by
  Corollary~\ref{cor-discretely-trivial}.
\end{remark}

With straightforward modifications to Lemma~\ref{lem-gx-top-prin} and
Cor\-ol\-lary~\ref{cor-gx-essen-prin} using
Lem\-ma~\ref{lem-discretely-trivial}, we obtain the following analogue
of Theorem~\ref{thm-top-ess-prin}.

\begin{thm}
  \label{thm-str-top-ess-prin}
  Suppose that $G$ and $H$ are equivalent groupoids.
  \begin{enumerate}
  \item $G$ is strongly topologically principal if and only if $H$ is
    strongly topologically principal.
  \item $G$ is strongly essentially principal if and only if $H$ is
    strongly essentially principal.
  \end{enumerate}
\end{thm}

\section{Isotropy and Orbits}
\label{sec:orbits-isotropy}

\begin{remark}
  \label{rem-iso-props} One important consequence of
  Proposition~\ref{prop-iso-iso} is that generic properties of the
  isotropy of equivalent groupoids are preserved.  For example, if $G$
  has abelian isotropy, then the same is true of any groupoid $H$
  equivalent to $G$.   Naturally, similar statements can be made if
  the isotropy groups are all  amenable, GCR, or CCR.
\end{remark}

If $G$ is a groupoid, then topology of the orbit space
$G\backslash \go=\go/G$ plays a prominent role is deciphering the
structure of $\cs(G)$.  Hence the following will be useful.

\begin{prop}
  \label{prop-orbit-space} If $G$ and $H$ are equivalent groupoids
  then the orbit spaces $G\backslash \go$ and $H\backslash \ho$ are
  homeomorphic.   Moreover if $Z$ is a $(G,H)$-equivalence, then the
  map $G\cdot r(z)\mapsto s(z)\cdot H$ is well-defined and gives such
  a homeomorphism.
\end{prop}

\begin{proof}
  It suffices to prove the last assertion which is exactly
  \cite{wil:toolkit}*{Lemma~2.41}.  We provide the proof for
  convenience.  If $Z$ is a $(G,H)$-equivalence, then we get a
  $G$-equivariant homeomorphism $\bar r:Z/H\to \go$.  Since the orbit
  maps are open by \cite{wil:toolkit}*{Proposition~2.12}, we get a
  homeomorphism $\underline r :G\backslash (Z/H) \to G\backslash
  \go$.  Similarly, we get a homeomorphism $\underline s:(G\backslash 
  Z)/H\to \ho/H$ such that $\underline s\bigl((G\cdot z)\cdot H\bigr)
  =s(z)\cdot H$.  Since $(G\cdot z)\cdot H\mapsto G\cdot (z\cdot H)$
  gives a homeomorphism of $(G\backslash Z)/H$ onto $G\backslash
  (Z/H)$, this suffices.
\end{proof}

Recall that the action of $G$ on $\go$ is \emph{minimal} if every
orbit $G\cdot u$ is dense in $\go$.   Since this is equivalent to
saying that every point in $G\backslash \go$ is dense, the next
corollary is immediate.

\begin{cor}
  \label{cor-minimal} Suppose that $G$ and $H$ are equivalent
  groupoids.  Then $G$ acts minimally on $\go$ if and only if $H$ acts
  minimally on $\ho$.
\end{cor}

In the same spirit, recall that a topological space is
$T_{0}$ if distinct points have distinct closures, $T_{1}$ if points
are closed, and $T_{2}$ if the space is Hausdorff.  Applied to the
orbit space $G\backslash \go$, the later is $T_{0}$ exactly when
distinct orbits in $\go$ have distinct closures.  It is $T_{1}$ when orbits
are closed in $\go$.  Now we can apply
Proposition~\ref{prop-orbit-space} to obtain the following.

\begin{cor}
  \label{cor-tone-orbit-space}
  Suppose that $G$ and $H$ are equivalent groupoids.  Then
  $G\backslash \go$ is
  $T_{0}$ (resp., $T_{1}$, resp., $T_{2}$) if and only if
  $H\backslash\ho$ is $T_{0}$ 
  (resp., $T_{1}$, resp., $T_{2}$). 
\end{cor}

% \begin{remark}
%   It is possible to give a direct proof of
%   Corollary~\ref{cor-tone-orbit-space} using blow-ups, but it is more
%   straightforward to simply refer to
%   Proposition~\ref{prop-orbit-space}.
% \end{remark}

Recall from \cite{wil:toolkit}*{\S3.4} that the space $\so$ of closed
subgroups of a groupoid $G$ has a locally compact topology which is
the relative topology coming from the Fell topology on the compact
space $\mathcal C(G)$ of closed subsets of $G$.  We say that the
isotropy of $G$ is continuous at $u\in\go$ if $G(u_{n})\to G(u)$ in
$\so$ for any net $(u_{n})$ in $\go$ converging to $u$ in $\go$.

\begin{prop}
  \label{prop-cts-iso} Suppose that $Z$ is a $(G,H)$-equivalence.
  Then the isotropy of $G$ is continuous at $r(z)$ if and only if the
  isotropy of $H$ is continuous at $s(z)$.
\end{prop}

Before proceeding with the proof,
we need a preliminary result involving blow-ups.   For a blow-up
$\gx$, the isotropy groups are related as follows:
\begin{equation}
  \label{eq:4a}
  \gx(x)=\set{(x,\gamma,x)\in\gx:\gamma\in G(f(x))}.
\end{equation}

\begin{lemma}
  \label{lem-cts-iso} Suppose that $f:X\to\go$ is a continuous open
  surjection.  Then the isotropy of $G$ is continuous at $f(x)$ if and
  only if the isotropy of $\gx$ is continuous at $x$.
\end{lemma}
\begin{proof}
  Let $u=f(x)$.  Suppose that the isotropy is continuous at $u$ and
  that $(x_{i})$ is a net converging to $x$ in $X$.  We claim that
  $\gx(x_{i})\to \gx(x)$.  We apply the criteria in
  \cite{wil:toolkit}*{Lemma~3.22}.  We can assume that we have already
  passed to a subnet and relabeled.  Suppose that
  $(x_{i},\gamma_{i},x_{i}) \in \gx(x_{i})$ and
  $(x_{i},\gamma_{i},x_{i})\to (y,\gamma,z)$ in $\gx$.  Then we must
  have $y=x=z$.  Furthermore $\gamma_{i}\in G(f(x_{i}))$ and
  $\gamma_{i}\to \gamma$.  By assumption, $G(f(x_{i}))\to G(f(x))$.
  Hence $\gamma\in G(f(x))$.  Thus part~(a) of
  \cite{wil:toolkit}*{Lemma~3.22} is satisfied.

  For part~(b), suppose $(x,\gamma,x)\in\gx(x)$.  Then
  $\gamma\in G(f(x))$ and we still have $G(f(x_{i}))\to G(f(x))$.  Hence
  we can pass to a subnet, relabel, and assume that there are
  $\gamma_{i}\in G(f(x_{i}))$ with $\gamma_{i}\to \gamma$.  Then
  $(x_{i},\gamma_{i},x_{i})\to (x,\gamma,x)$ and we have shown that
  the isotropy is continuous at $x$.

  Conversely, suppose that the isotropy of $\gx$ is continuous at $x$
  and that $u_{i}\to u=f(x)$ in $\go$.  We want to show that
  $G(u_{i})\to G(u)$.  Again, we can assume that we have passed to a
  subnet and relabeled.  Since $f$ is open, we can pass to a further
  subnet if necessary, and assume that there are $x_{i}\in X$ such
  that $x_{i}\to x$ and $f(x_{i})=u_{i}$.  Then by assumption,
  $\gx(x_{i})\to \gx(x)$.  If $\gamma_{i}\in G(u_{i})$ and
  $\gamma_{i}\to \gamma$ in $G$, then
  $(x_{i},\gamma_{i},x_{i})\to (x,\gamma,x)$ in $\gx$.  Hence,
  $(x,\gamma,x)\in\gx(x)$ and $\gamma\in G(u)$.  If $\gamma\in G(u)$,
  then $(x,\gamma,x)\in \gx(x)$ and we can pass to a subnet, relabel,
  and find $(x_{i},\gamma_{i},x_{i}) \in\gx(x_{i})$ such that
  $(x_{i},\gamma_{i},x_{i})\to (x,\gamma,x)$.  But then,
  $\gamma_{i}\to\gamma$ and $G(u_{i})\to G(u)$ as claimed.
\end{proof}

\begin{proof}[Proof of Proposition~\ref{prop-cts-iso}]
  % Let $G[Z]$ be the blow-up of $G$ with respect to the moment map
%   $r_{Z}:Z\to\go$ and let $H[Z]$ be the blow-up of $H$ with respect to
%   the moment map $s_{Z}:Z\to\ho$.   Then as in the proof of
%   \cite{wil:toolkit}*{Theorem~2.52}, we get an isomorphism
%   $\phi:G[Z]\to H[Z]$ given by
%   \begin{equation}
%     \label{eq:2aa}
%     \phi(z,\gamma,w)=(z,\tau_{H}(z,\gamma\cdot w),w).
%   \end{equation}
%   Since $\phi(z,r_{Z}(z),z)=(z,s_{Z}(z),z)$,
%   the isotropy of $G[Z]$ is continuous at $(z,r_{Z}(z),z)$ if
% and only the isotropy of $H[Z]$ is  continuous at $(z,s_{Z}(z),z)$.
% Now the proposition follows from Lemma~\ref{lem-cts-iso}.
  The result now follows using the proof of
  \cite{wil:toolkit}*{Theorem~2.52} exactly as in the proof of
  Proposition~\ref{prop-discretely-trivial}. 
\end{proof}

\begin{cor}
  \label{cor-cts-iso} Suppose that $G$ and $H$ are equivalent
  groupoids.   Then $u\mapsto G(u)$ is continuous on $\go$ if and only
  if $v\mapsto H(v)$ is continuous on $\ho$.
\end{cor}

As an application of these ideas, we recall a condition on a groupoid
with abelian isotropy the authors
introduced in \cite{wykwil:iumj22}---called \emph{proper modulo the
  isotropy}---generalizing a proper groupoid with abelian
  isotropy.  This allowed us to describe the primitive ideal space
$\operatorname{Prim}(\cs(G))$ as a topological space, generalizing
work of \cite{echeme:em11} and \cite{neu:phd11} in the case of abelian
isotropy.

We outline the details from  \cite{wykwil:iumj22}*{\S9}. We let
$R\subset \go\times\go$ be the image of $\pi:G\to\go\times\go$ where
$\pi(\gamma) = \bigl(r(\gamma),s(\gamma)\bigr)$.   Then with the
relative product topology, $R$ is a topological groupoid which can
violate our standing assumptions as $R$ need not be locally
compact, nor need it have open range and source maps.   However, if
$G\backslash \go$ is Hausdorff, then $R$ is at least closed in
$\go\times\go$.  If $G$ has abelian
isotropy, then there is a well-defined action of $R$ on $\so$ given by
$\pi(\gamma)\cdot H=\gamma\cdot H=\gamma H \gamma^{-1}$.  We say that
\emph{$G$ is proper modulo its isotropy} if $G$ has abelian isotropy,
$G\backslash \go$ is Hausdorff,
and $R$ acts continuously on $\so$.  As observed in
\cite{wykwil:iumj22}*{Example~9.5}, there is a large class of groupoids,
such as those studied in \cites{iksw:jot19,ikrsw:jfa20,ikrsw:nzjm21},
that are proper modulo their isotropy.

For the remainder of this section we assume that $G$ has abelian
isotropy.   Then if $H$ is equivalent to $G$, it has abelian isotropy
as well by Proposition~\ref{prop-iso-iso}. 

\begin{thm}
  \label{thm-proper-mod-iso} Suppose that $G$ and $H$ are equivalent
  groupoids.   Then $G$ is proper modulo its isotropy if and only if $H$
  is proper modulo its isotropy.
\end{thm}

Since Corollary~\ref{cor-tone-orbit-space} implies that
$G\backslash \go$ is Hausdorff if and only if $H\backslash \ho$
is, to prove the theorem it will suffice to prove the following lemma
and appeal to Remark~\ref{rem-std} as usual.

\begin{prop}\label{prop-key}
  Suppose that $f:X\to\go$ is a
  continuous, open surjection.   Let $R=\pi(G)$ as above, and let
  $R'=\pi'(\gx)$ where $\pi':\gx\to X\times X$ is the corresponding
  map for $\gx$.   Then $R$ acts continuously on $\so$ if and only if $R'$
  acts continuously on $\so'$ where $\so'$ is the space of closed
  subgroups of $\gx$.
\end{prop}

Before proceeding with the proof of the lemma, we introduce some
notation and make a few observations.  Let $p_{0}:\so\to\go$ be given
by $p_{0}(H)=u$ when $H\subset G(u)$.  Let $p_{0}':\so'\to X$
be the corresponding map for $\gx$.   Using \eqref{eq:4a}, there is a
map $\rho:\so'\to\so$ such that if $p_{0}'(H')=x$, then
\begin{equation}
  \label{eq:3}
  H'=\set{(x,\gamma,x):\gamma\in \rho(H')}.
\end{equation}
Note that $p_{0}(\rho(H'))=f(p_{0}'(H'))$.  Using
\cite{wil:toolkit}*{Lemma~3.22} it is not hard to establish the
following.

\begin{lemma}
  \label{lem-conv} A net $H_{n}'\to H'$ in $\so'$ if and only if
  $\rho(H_{n}')\to \rho(H')$ in $\so$ and $p_{0}'(H_{n}')\to
  p_{0}'(H')$. 
\end{lemma}

\begin{lemma}
  \label{lem-rho-phi-prop} The maps $\rho:\so'\to\so$ and $\cphi:\gx\to
  G$ are continuous, open surjections.
\end{lemma}
\begin{proof}
  To see that $\rho$ is open, we use Fell's criterion.  Suppose that
  $H_{n}\to \rho(H')$.  Let $u_{n}=p_{0}(H_{n})=f(p_{0}'(H'))$.  Since
  $f$ is open, we can pass to a subnet, relabel, and assume that there are
  $x_{n}\to p_{0}'(H')$ such that $f(x_{n})=u_{n}$.  Let
  $H_{n}'=\set{(x_{n},\gamma,x_{n}):\gamma\in H_{n}}$.  Now we can use
  Lemma~\ref{lem-conv} to show that $H_{n}'\to H'$.  Since
  $\rho(H_{n}')=H_{n}$ by construction, this proves openness.

  The proof of continuity for $\rho$ is similar, but easier, as are
  the assertions about~$\cphi$.
\end{proof}

Observe that if $H'\in\so'$ and $\alpha\in\gx$ is such that
$s(\alpha)=p_{0}'(H')$, then
\begin{equation}
  \label{eq:7a}
  \alpha\cdot H'=H''
\end{equation}
where $\rho(H'')=\cphi(\alpha)\cdot \rho(H')$.

\begin{proof}[Proof of Proposition~\ref{prop-key}]
  Suppose that $R$ acts continuously on $\so$.  We want to show that
  $R'$ acts continuously on $\so'$.  To this end, suppose that
  $H_{n}'\to H'_{0}$ in $\so'$ with $x_{n}=p_{0}'(H_{n}')$.  Suppose
  we also have
  $\pi'(\alpha_{n})=(y_{n},x_{n})\to \pi'(\alpha_{0})=(y_{0},x_{0})$.
  We want to see that
  $\pi'(\alpha_{n})\cdot H_{n}'\to \pi'(\alpha_{0})\cdot H_{0}'$.  Let
  $H_{n}=\rho(H_{n}')$.  Then the continuity of $\rho$ implies that
  $H_{n}\to H_{0}$.  Moreover,
  $u_{n}=f(x_{n})=p_{0}(H_{n})\to u_{0}=f(x_{0})=p_{0}(H_{0})$.  Hence
  $\pi(\cphi(\alpha_{n}))\cdot H_{n}\to \pi(\cphi(\alpha_{0}))\cdot
  H_{0}$.  Since
  $\rho(\pi'(\alpha_{n})\cdot H_{n}')=\pi(\cphi(\alpha_{n}))\cdot
  \rho(H_{n}')$, Lemma~\ref{lem-conv} implies that
  $\pi'(\alpha_{n})\cdot H_{n}'\to \pi'(\alpha_{0})\cdot H_{0}'$ as
  required.

  Conversely, assume that $R'$ acts continuously  on $\so'$.  To show
  that $R$ then acts continuously on $\so$, we assume $H_{n}\to H_{0}$
  with $p_{0}(H_{n})=u_{n}$.   We also suppose
  $\pi(\gamma_{n})=(v_{n},u_{n})\to (v_{0},u_{0})$.  We need to
  establish that $\pi(\gamma_{n})\cdot H_{n}\to \pi(\gamma_{0})\cdot
  H_{0}$.   For this, it suffices to see that any subnet has a subnet
  converging to $\pi(\gamma_{0})\cdot H_{0}$.   Since $f$ is open, we
  can pass to a subnet, relabel, and assume that there are $x_{n}\to
  x_{0}$ in $X$ such that $f(x_{n})=u_{n}$.   Passing to still another
  subnet, we can also assume that there are $y_{n}\to y_{0}$ in $X$
  such that $f(y_{n})=v_{n}$.  Then
  $\alpha_{n}=(y_{n},\gamma_{n},x_{n})\in\gx$ and $\alpha_{n}\to
  \alpha_{0}$.  Let $H_{n}'=\set{(x_{n},\gamma,x_{n}):\gamma\in
    H_{n}}$.  Then $H_{n}'\to H_{0}'$ with $p_{0}'(H_{n}')=x_{n}$.
  By assumption, $\pi'(\alpha_{n})\cdot H_{n}'\to
  \pi'(\alpha_{0})\cdot H_{0}'$.  Since $\cphi(\alpha_{n})=\gamma_{n}$
  by construction, the continuity of $\rho$ implies
  that $\pi(\gamma_{n})\cdot H_{n}\to \pi(\gamma_{0})\cdot H_{0}$.
  This completes the proof.
\end{proof}

\section{Haar Systems and Integrable Groupoids}
\label{sec:integrable-groupoids}

In this section, we will be primarily working with second countable
groupoids with Haar systems.  The following is the main result in
\cite{wil:pams16}.

\begin{thm}[\cite{wil:pams16}*{Theorem~2.1}]
  \label{thm-haar} Suppose that $G$ and $H$ are equivalent second
  countable groupoids.  Then $G$ has a Haar system if and only if $H$
  has a Haar system.
\end{thm}

If $(G,\lambda)$ is a groupoid with a  Haar system
$\lambda=\set{\lambda_{u}}_{u\in\go}$, then 
building on \cite{rie:em04} and \cite{hue:iumj02}, in
\cite{clahue:pams08}*{Definition~3.1}, Clark and an Huef
define a groupoid $G$ to be \emph{integrable} if for
every compact set $N\subset \go$,
\begin{equation}
  \label{eq:10}
  \sup_{u\in N} \lambda^{u}(s^{-1}(N))<\infty.
\end{equation}

As stated, integrability does not appear to be a purely topological
property since \emph{a priori}, it depends on a choice of Haar system.
However, if $G$ is a second countable principal groupoid, then Clark
and an Huef prove that $G$ is integrable if and only if
$\cs(G,\lambda)$ has bounded trace \cite{clahue:pams08}*{Theorem~4.4}.
It follows from the Equivalence Theorem that the Morita equivalence
class of $\cs(G)$ is invariant under any choice of a Haar system
\cite{wil:toolkit}*{Proposition~2.74}.  Since the property of having
bounded trace is preserved by Morita equivalence
\cite{hrw:pams07}*{Proposition~7}, it follows that a principal
groupoid $G$ is integrable if and only if \eqref{eq:10} holds for
some, and hence any, Haar system on $G$.  Furthermore, if $G$ is
equivalent to $H$ and if $H$ admits a Haar system, then so does $G$
by Theorem~\ref{thm-haar}.  In particular, we have the following.

\begin{prop}
  \label{prop-bounded-trace} Suppose that $G$ is a second countable
  principal groupoid admitting a Haar system, and that $H$ is a second
  countable groupoid equivalent to $G$.  If $G$ is integrable, then
  $H$ admits a Haar system and $H$ is integrable.
\end{prop}

\begin{remark}
  It would be interesting to determine if the property of being
  integrable is independent of the choice of a Haar system in the
  general case.   Even more tempting, it would be nice to find a
  purely topological criterion that does not depend on the existence
  of a Haar system.   One possibility is suggested by
  \cite{clahue:pams08}*{Definition~3.6} and
  \cite{clahue:pams08}*{Proposition~3.11}.   Unfortunately, the
  converse of  \cite{clahue:pams08}*{Proposition~3.11} is only known
  in the principal case.
\end{remark}

\section{Twists}
\label{sec:t-groupoids}

The notion of a \emph{twist} or a \emph{$\T$-groupoid} originated in
\cite{kum:cjm86} and has been extensively studied
\cites{muhwil:ms92,muhwil:jams04,muhwil:plms395}.  Recall that a twist
$E$ over $G$ is given by a central groupoid extension 
\begin{equation}
  \label{eq:50}
  \begin{tikzcd}
    \go\times \T\arrow[r,"\iota"] &E \arrow[r,"j",two heads]& G
  \end{tikzcd}
\end{equation}
where $\iota$ and $j$ are continuous
groupoid homomorphisms such that $\iota$ is
a homeomorphism onto its range, $j$ is an open surjection inducing a
homeomorphism of the unit space of $E$ with $\go$, and with kernel
equal to the range of $\iota$.  We identify the unit space of $E$ with $\go$. 
Furthermore
\begin{equation}
  \label{eq:51}
  \iota(r(e),t)e=e\iota(s(e),t)\quad\text{for all $e\in E$ and $t\in \T$.}
\end{equation}
Note that $E$ becomes a
principal $\T$-bundle with respect to the action $t\cdot
e=\iota(r(e),t)e$. 

Conversely, we can think of a twist as a groupoid $E$ admitting a
free left $\T$-action that is compatible with the groupoid structure
as in \cite{muhwil:plms395}*{\S3}.   Specifically, we have the
following observation.
\begin{lemma}
  \label{lem-muhwil} Suppose $\T$ acts freely on a groupoid $E$ such
  that $r(t\cdot e)=r(e)$, $s(t\cdot e)=s(e)$, and
  $(t\cdot e)(t'\cdot f)=(tt')\cdot ef$.  Let $G$ be the orbit space
  $\T\backslash E$, and let $j:E\to G$ the orbit map.  Then $G$ is a
  locally compact groupoid with respect to the operations
  $j(e)j(f)=j(ef)$ with $(e,f)\in E^{(2)}$, and $j(e)^{-1}=j(e^{-1})$.
  Furthermore $E$ is twist over $G$ with $\iota(u,t)=t\cdot u$.
\end{lemma}
\begin{proof}[Sketch of the Proof]
  Since $\T$ is compact, $G$ is locally compact Hausdorff.   It is
  routine to verify that $G$ is a groupoid with $\go$ identified with
  $E^{(0)}$ via $u\mapsto j(u)$, and that \eqref{eq:50} is exact.
  Since $\T$ is compact, it
  follows that $r_{G}(j(e))=r_{E}(e)$ is open as is $s_{G}$.
\end{proof}

If $E$ is a twist acting on the left of a space $Z$, then we get a
$\T$-action on the left of $Z$ by $t\cdot z=\iota(r(z),t)\cdot z$.
Similarly, if $Z$ is a right $E$-space, we get a right $\T$-action on
$Z$ by $z\cdot t=z\cdot \iota(s(z),t)$.  In particular, if $Z$ is an
equivalence between two twists, then it is both a left $\T$-space and
a right $\T$-space.

\begin{definition}[\cite{muhwil:plms395}*{Definition~3.1}]
  Suppose that $E$ and $E'$ are twists over $G$ and $G'$,
  respectively.  We say that a $(E,E')$-equivalence $Z$ is an
  \emph{$(E,E')$-twist equivalence over $(G,G')$} if $t\cdot z=z\cdot t$ for all
  $t\in\T$ and $z\in Z$.
\end{definition}

\begin{remark}
  \label{rem-mw-name} In \cite{muhwil:plms395} a $(E,E')$-twist
  e\-qui\-va\-lence was called a $(E,E')$-$\T$-e\-qui\-va\-lence.  In
  \cite{muhwil:plms395}*{Theorem~3.2}, it is shown that if $E$ and
  $E'$ are second countable twists such that $G$ and $G'$ have Haar
  systems, then the restricted groupoid \cs-algebras $\cs(G;E)$ and
  $\cs(G';E')$ are Morita equivalent.
\end{remark}

Let $Z$ be a $(E,E')$-twist equivalence over $(G,G')$.   Let $\ZT$ be
the locally compact Hausdorff quotient $\T\backslash Z=Z/\T$.  If
$z\in Z$, we let $[z]=z\cdot\T=\T\cdot z$.

Suppose $e,f\in E$ and $z,w\in Z$ are such that
$j(e)=j(f)$ and $[z]=[w]$.  Then there are $t,t'\in\T$ such
that $f=t\cdot e$ and $w=t'\cdot z$.  Then
\begin{equation}
  \label{eq:8a}
  [f\cdot w]=[(t\cdot e)\cdot (t'\cdot z)]=[t\cdot (e\cdot z) \cdot
  t']=[e\cdot z].
\end{equation}
With a similar observation when $j'(e')=j'(f')$ in $G'$, we have a
left $G$-action on $\ZT$ and a right action of $G'$ on $\ZT$ given
by
\begin{equation}
  \label{eq:7}
  j(e)\cdot [z]=[e\cdot z]\quad\text{and}\quad [z]\cdot j'(e')=[z\cdot
  e'],
\end{equation}
respectively.  Note that if $r_{Z}$ is the moment map for the left
$E$-action on $Z$, then the moment map for the left $G$-action on
$\ZT$ is given by $r_{\ZT}([z])=r_{Z}(z)$.   Usually, we will abuse
notation and simply write $r$ for these maps and $s$ for the
corresponding moment maps for the right actions.

\begin{prop}
  \label{prop-ggp-equi} Let $Z$ be a $(E,E')$-twist equivalence over
  $(G,G')$, and let $\ZT$ be the quotient $(G,G')$-space as above.
  Then $\ZT$ is a $(G,G')$-equivalence.  In particular, $G$ and $G'$
  are equivalent groupoids.
\end{prop}
\begin{proof}
  To see that the left $G$-action is continuous, suppose that
  $j(e_{i})\to j(e)$ in $G$ and that $[z_{i}]\to [z]$ in $\ZT$ with
  $s(e_{i})=r(z_{i})$.  It will suffice to see that every subnet of
  $([e_{i}\cdot z_{i}])$ has a subnet converging to $[e\cdot
  z]$. Since the orbit maps are open, after
  passing to a subnet, we can pass to a further subnet, relabel, and
  assume that $e_{i}\to e$ while $z_{i}\to z$.  But then
  $e_{i}\cdot z_{i}\to e\cdot z$.  Therefore the left $G$-action is
  continuous.  The proof for the right $G'$-action is similar.

  To see that $\ZT$ is a $(G,G')$-equivalence, we verify (E1)--(E5)
  Definition~\ref{def-equivalence}.
  To see that the left $G$-action is proper, suppose that $[z_{i}]\to
  [z]$ and that $j(e_{i})\cdot [z_{i}]\to [w]$.  After passing to a
  subnet and relabeling, we can assume that $z_{i}\to z$ while
  $e_{i}\cdot z_{i}\to w$.   Then $(e_{i})$ must have a convergent
  subnet which implies that $(j(e_{i}))$ does as well.   This suffices
  by \cite{wil:toolkit}*{Proposition~2.17}.   Since the $G$-actions is
  clearly free, $\ZT$ is a free and proper left $G$-space.   A similar
  argument shows that it is a free and proper right $G'$-space.  This
  establishes (E1) and (E2).

  Furthermore, these actions commute: $\bigl( j(e)\cdot [z]\bigr)\cdot
  j'(e') =[e\cdot z]\cdot j'(e')=[(e\cdot z)\cdot e']= [e\cdot (z\cdot
  e')] = j(e) \cdot \bigl([z]\cdot j'(e')\bigr)$. So (E3) holds.

  To see that the moment map for the $G$-action is open, suppose that
  $u_{i}\to r(e)=r(j(e))$.    Since the moment map for the $E$-action
  is open, we can pass to a subnet, relabel, and assume that $e_{i}\to
  e$ with $r(e_{i})=u_{i}$.   But then $j(e_{i})\to j(e)$.  Similarly,
  the moment map for the $G'$ action is open.

  Moreover if $r([z])=r([z'])$, then $r(z)=r(z')$ and there is a $e'$
  such that $z=z'\cdot e'$.  But then $[z]=[z']\cdot j'(e')$ and $r$
  factors through a homeomorphism of $\ZT/G'$ onto $\go$.    This
  establishes (E4), and (E5) is proved similarly. 
\end{proof}

Our next result is essentially a reworking of
\cite{muhwil:plms395}*{Theorem~3.5}.  We sketch the details here.

\begin{thm}
  \label{thm-muhwil-3.5} Suppose that
\begin{equation}
  \label{eq:9b}
  \begin{tikzcd}
    (G')^{(0)}\times \T\arrow[r,"i'"]& E'\arrow[r,"j'",two heads]&G'.
  \end{tikzcd}
\end{equation}
is a twist over $G'$.  If $E$ is equivalent to $E'$ and if $Z$ is an
$(E,E')$ equivalence, then
$E$ admits a principal left $\T$-action so
that $E$ is a twist over a groupoid $G$ with unit space $Z/E'$.
Furthermore, with respect to this twist structure on $E$, $Z$ is a
$(E,E')$-twist equivalence over $(G,G')$. 
\end{thm}

\begin{proof}
Using \cite{wil:toolkit}*{Lemma~2.42 and Remark~2.43} we can
assume that $E=(Z*_{s}Z)/E'$.  Recall that
the groupoid structure on $E$ is given
as follows.   We have  $[z,w]^{-1}=[w,z]$.  Also
$[z,w]$ and $[x,y]$ are composible
if $w\cdot E'=x\cdot E'$.   Then there is a unique $e'\in E'$ such
that $x=w\cdot e'$ and $[z,w][x,y]=[z\cdot e',y]$.  Then we can
identify the unit space of $E$ with $Z/E'$ and then
$r([z,w])=z\cdot E'$
and $s([z,w])=w\cdot E'$.   It is common practice to summarize the
composition law as simply 
$[z,w][w,p]=[z,p]$ with the gyrations
\begin{equation}
  \label{eq:13}
  [z,w][x,y]=[z,w][w\cdot e',y ][z,w][w,y \cdot
(e')^{-1}]=[z,y\cdot (e')^{-1}]=[z\cdot e',y].
\end{equation}
left implicit.

With this identification of $E$, the moment map $r$ on $Z$ is
given by $r(z)=x\cdot E'$.  The left action of $E$ on $Z$ is given by
$[z,w]\cdot w=z$ (with a similar understanding as with the
simplified composition law).

Since $\T$ is abelian, we can define a left
$\T$-action on $Z*_{s}Z$ using the right $\T$-action on $Z$ induced by
$E'$: $t\cdot (z,w)=(z\cdot t,w)$.  If $e'\in E'$ with $r(e')=s(z)$,
then
\begin{align}
  \label{eq:4}
  ((z\cdot t)\cdot e',w\cdot e')
  &= (z\cdot \iota'(s(z),t)e',w\cdot e') 
  = (z\cdot (e'\iota'(s(e'),t),w\cdot e') \\
  &= ((z\cdot e')\cdot t,w\cdot e').
\end{align}
It follows that we get a left action of $\T$ on $E$ given by
\begin{equation}
  \label{eq:8}
  t\cdot [z,w]=[z\cdot t,w].
\end{equation}
If $[z\cdot t,w]=[z,w]$, then $t=1$ since the $E'$-action on $Z$ is
free.   Hence $E$ becomes a principal $\T$-space.  Furthermore,
\begin{equation}
  \label{eq:11}
  [z\cdot t,w]=[z\cdot \iota'(s(z),t),w]=[z,w\cdot
  \iota'(s(w),\overline t)]=[z,w\cdot \overline t].
\end{equation}
Then
\begin{align}
  \label{eq:12}
(t\cdot [z,w])(s\cdot [w,p]) 
&= [z\cdot t,w][w,p\cdot \overline s] 
=[z\cdot t,p\cdot \overline s]=[z\cdot (ts),p]\\
&=(ts)\cdot
\bigl([z,w][w,p]\bigr). 
\end{align}
Thus the $\T$-action on $E$ is compatible with the groupoid structure
and $E$ is a twist over the quotient $G=\T\backslash E$ as above.  It
particular, $\iota:Z/E'\times \T\to E$ is given by $\iota(z\cdot
E',t)=t\cdot [z,z]=[z\cdot t,z]$
Therefore if $t\in \T$ and $z\in Z$, we have
\begin{align}
  \label{eq:14}
  t\cdot z
  &=\iota(z\cdot E',t)\cdot z = [z\cdot t,z]\cdot z
    = z\cdot t.
\end{align}
That is $Z$ is a $(E,E')$-twist equivalence.
\end{proof}

Note that in Theorem~\ref{thm-muhwil-3.5}, the groupoids $G$ and $G'$
are equivalent by Proposition~\ref{prop-ggp-equi}.   Hence we can
apply our previous results in examples such as the following.
\begin{example}
  Suppose that $E$ is a twist over an essentially principal groupoid
  $G$.   If $E'$ is equivalent to $E$, then $E'$ is also a twist over an
  essentially principal groupoid $G'$.
\end{example}
Of course, we could replace ``essentially principal'' above with any
property preserved by equivalence.

%%%%%%%%%%%%%%%%%%%%%%%%%%%%%%%%%%%%%%%%%%
%%%%% End Matter

\appendix

\section{\'Etale Groupoids and Equivalence}
\label{sec:etale}

We have already observed in Example~\ref{ex-not-etale} that the
property of a groupoid being \'etale 
need not extend to arbitrary push-outs, let alone equivalent
groupoids.  To obtain some interesting specific examples, we can
appeal to 
\cite{rie:pspm82}*{Situation~4}. 

\begin{example}
  \label{ex-sit-8}
Let $H$ be a closed subgroup of $G$ and let $P$ be any right $H$-space.  Let
$Z=G\times P$.   Then $(g,p)\cdot h=(gh,P\cdot h)$ is a free and
proper $H$-space (because $H$ acts freely and properly on the right of
$G$).   Similarly, $g'\cdot (g,p)=(g'g,p)$ is a free and proper left
$G$-action.    Then Green's Symmetric Imprimitivity Theorem
(\cite{wil:crossed}*{Corollary~4.11}) implies that
\begin{equation}
  \label{eq:9}
  C_{0}(G\backslash Z)\rtimes_{\rt}H\quad\text{and}\quad C_{0}(Z/H)\rtimes_{\lt}G
\end{equation} 
are Morita equivalent.  In fact, in the second countable case, we can
also prove this using the observation that the action groupoids
$G\rtimes Z/H$ for the left $G$-action and
  $G\backslash Z \rtimes H$ for the right $H$-action are equivalent
(\cite{wil:toolkit}*{Example~2.34}), and then applying the
Equivalence Theorem \cite{wil:toolkit}*{Theorem~2.70}.

In particular, if $H$ is a discrete subgroup of $G$, then
$G\backslash Z \rtimes H$ is \'etale.   But if $G$ is not
discrete, then $G\rtimes Z/H$ is not.\footnote{This is because
  $(G\rtimes Z/H)^{(0)} = \sset 
  e\times Z/H$ is not open in $G\times Z/H$.}

For a specific example, we can let $G=\R$ and $H=\Z$.  Then we let $H$
act on $P=\T$ by an irrational rotation.  Since $G\backslash M$ is can
be identified with $P$, $C_{0}(G\backslash Z)\rtimes_{\rt}H$
$C_0(G\backslash Z)$ is an irrational rotation algebra.  But $Z/H$
can be identified with the torus $\T^{2}$ and the $\R$-action is the
flow along an irrational angle.
\end{example}

Example~\ref{ex-not-etale} shows that the push-out of an \'etale
groupoid can fail to be \'etale.   It is also the case that the
push-out of a non-\'etale groupoid is \emph{never} \'etale.

\begin{lemma}
  \label{lem-gx-converse} Suppose that $f:X\to \go$ is a continuous
  open surjection and that $G$ is not \'etale.   Then $\gx$ is not
  \'etale.   
\end{lemma}
\begin{proof}
  Since we are assuming that $r:G\to\go$ is open, $G$ not \'etale
  implies that $\go$ is not open in $G$
  \cite{wil:toolkit}*{Lemma~1.26}.   Thus there is a sequence
  $(\gamma_{n}) \subset G\setminus \go$ such that $\gamma_{n }\to
  u\in\go$.
  Let $f(x)=u$.  Since $r(\gamma_{n})\to u$, we can pass to a subnet,
  relabel, and assume that there are $x_{n}\in X$ with $x_{n}\to x$
  and $f(x_{n})=r(\gamma_{n})$.   Similarly, after passing to a subnet
  and relabeling, there are $y_{n}\in G$
  such that $y_{n}\to x$ and $f(y_{n})=s(\gamma_{n})$.   Then
  $(x_{n},\gamma_{n},y_{n})\to (x,u,x)$ in $\gx$.  Since
  $(x_{n},\gamma_{n},y_{n})\notin \gx^{(0)}$, it follows that
  $\gx^{(0)}$ is not open in $\gx$.   Hence $\gx$ is not \'etale.
\end{proof}

\def\noopsort#1{}\def\cprime{$'$} \def\sp{^}

\end{document}